\documentclass[11pt]{amsart}

\usepackage[utf8]{inputenc}
\usepackage{lmodern}
\usepackage[T1]{fontenc}
\usepackage[textwidth=6.2in, textheight=7.8in, hcentering]{geometry}
\usepackage[pagebackref=true,breaklinks=true,colorlinks]{hyperref}
\usepackage{amssymb}
\usepackage{color}
\usepackage[all,cmtip]{xy}
\usepackage{amsmath}
\usepackage{amsthm}
\usepackage{amscd}
\usepackage{amsfonts, enumitem}
\usepackage{tikz}
\usetikzlibrary{arrows,automata}
\usetikzlibrary{matrix,arrows,decorations.pathmorphing}

\DeclareMathOperator{\Aut}{Aut}
\DeclareMathOperator{\act}{act}
\DeclareMathOperator{\ACT}{ACT}
\DeclareMathOperator{\Res}{Res}
\DeclareMathOperator{\Fix}{Fix}
\DeclareMathOperator{\Map}{Map}
\DeclareMathOperator{\lMap}{Inv_l^r}
\DeclareMathOperator{\rMap}{Inv_r^l}
\DeclareMathOperator{\InvE}{INV}
\DeclareMathOperator{\Isom}{Isom}
\DeclareMathOperator{\Ho}{Ho}

\DeclareMathOperator{\Set}{Sets}
\DeclareMathOperator{\ev}{\mathcal{E}}
\DeclareMathOperator{\evl}{\mathcal{E}_l^r}
\DeclareMathOperator{\evr}{\mathcal{E}_r^l}
\DeclareMathOperator{\bev}{\overline{\ev}}
\DeclareMathOperator{\A}{Br}
\DeclareMathOperator{\End}{End}

\newcommand{\bbZ}{\mathbb Z}
\newcommand{\bbN}{\mathbb N} 
\newcommand{\conj}{\sim_{I}}
\newcommand{\UU}{\mathcal{U}}
\newcommand{\VV}{\mathcal{V}}
\newcommand{\WW}{\mathcal{W}}
\newcommand{\bact}{\overline{\act}}
\newcommand{\bACT}{\overline{\ACT}}

\newtheorem{dummy}{anything}[section]
\newtheorem{theorem}[dummy]{Theorem}

\newtheorem{lemma}[dummy]{Lemma}
\newtheorem{proposition}[dummy]{Proposition}

\theoremstyle{definition}%%Change Theoremstyle
\newtheorem{definition}[dummy]{Definition}

\newtheorem{remark}[dummy]{Remark}

\begin{document}
	
	\title{Semigroup actions on sets and the Burnside ring}
	\author{Mehmet Akif Erdal}
	\thanks{This is the pre-print of an article published in 	Applied Categorical Structures. The final version is available online at: https://doi.org/10.1007/s10485-016-9477-4.}

	  \author{\" Ozg\"un \" Unl\"u }
	  \thanks{The second author is partially supported by T\"UBA-GEB\.IP/2013-22}

\address{Department of Mathematics,
	Bilkent University, 06800, Ankara, Turkey}
%\curraddr{}
\email{merdal@fen.bilkent.edu.tr}
%\thanks{}tor

	\begin{abstract}
	In this paper we discuss some enlargements of the category of sets with semigroup actions and equivariant functions. We show that these  enlarged categories possess two idempotent endofunctors. In the case of groups these enlarged categories are equivalent to the usual category of group actions and equivariant functions, and these idempotent endofunctors reverse a given action. For a general semigroup we show that these enlarged categories admit homotopical category structures defined by using these endofunctors and show that up to homotopy these categories are equivalent to the usual category of sets with semigroup actions. We finally construct the Burnside ring of a monoid by using homotopical structure of these categories, so that when the monoid is a group this definition agrees with the usual definition, and we show that when the monoid is commutative, its Burnside ring is equivalent to the Burnside ring of its Gr\"othendieck group. 
	
	\end{abstract}
\maketitle
\section{Introduction}
\label{section:Introduction}
In the classical terminology, the category of sets with (left) actions of a monoid corresponds to the category of functors from a monoid to the category of sets, by considering monoid as a small category with a single object. If we ignore the identity morphism on the monoid, it corresponds the category of sets with actions of a semigroup, which is conventionally used in applied areas of mathematics such as computer science or physics. For this reason we try to investigate our notions for semigroups, unless we need to use the identity element. In this note we only consider the actions on sets so that we often just write ``actions of semigroups" or ``actions of monoid" without mentioning ``sets". 

Actions of semigroups  appear quite often as mathematical models of progressive processes. In computer science, for example automata, or so called state machines, can be defined using semigroup actions. In physics a dynamical system can be seen as a semigroup action. An important problem in the theory of semigroup actions is reversibility of the actions. Reversible actions are particularly important when one considers applications. For example, in \cite{landauer} Landauer establishes the relation of reversibility of computation with energy consumption. Reversibility is also a fundamental issue in the theory of quantum state machines, since a quantum automaton has to be reversible. In dynamical systems the periodic attractors can be considered as reversible parts of the dynamical system.

In the theory of group actions, when a group $G$ is given, one often considers either left actions of $G$, or right actions of $G$, or if another group $H$ is given, one considers $(G,H)$-bisets, i.e. biactions of  $G$ and $H$ so that $G$ acts from the left and $H$ acts from the right. The categories of these actions are also well studied in the literature, see e.g. \cite{bouc}. The same distinction is also present in semigroup actions. For a given semigroup $I$, we define actions by fusing previous ideas and use some exotic notion of equivariance for biactions of a semigroup on a set, in a way that the biaction behaves like a single action, which generalizes the actions from one side, so that we no longer need to call them left, right or biactions, and we call them just ``actions".

We construct the category of actions in Section \ref{section:CategoriesofIsets}, see Lemma \ref{comp}, and we denote the category of all actions by $\ACT(I)$. For groups this category will be equivalent to the one defined in the usual way, so that when $I$ is a group the category of left $I$-sets (which is equivalent to right $I$-sets) is equivalent to $\ACT(I)$. In the theory of group actions when a left group action on a set is given, one can define a right action on the same set given by acting with the inverses of elements in the group; which is often called reverse action of the given left action. A similar construction exits for right actions as well. Due to lack of inverses these ``reverse action" constructions are not possible for actions of semigroups. On the other hand, generalizations are still possible for semigroup actions on sets by  considering our definition of actions. One of the objects of this paper is to define reverse actions so that they generalize the ones for groups. For a semigroup action on a set, these constructions are called ``reversing from left to right" and ``reversing from right to left", see Section \ref{subsection:InverseActionsOnSets}. Although the reverse actions have to be defined on different sets, it agrees with the above construction up to isomorphism when we consider group actions on sets, see Theorem \ref{l2r}. These constructions define two endofunctors on $\ACT(I)$, which will be called the reversing functors, $\lMap $ and $\rMap$, which are idempotent, see Theorem \ref{idempotent}. Composition of these functors will not be idempotent in general, but when we restrict our attention to finite $I$-sets it will be idempotent.

There are several other major advantages of these definitions of actions. For a semigroup $I$, the category $\ACT(I)$ has a subcategory denoted by $\bACT(I)$, whose objects are  actions which are ``reversible on one side" with equivariant maps between them, see Section \ref{section:CategoriesofIsets}. We show that this subcategory $\bACT(I)$ possesses a homotopical category structure in the sense of \cite{dwyer}, see Section \ref{homotoical}. In some respect, we can say this paper is initiative for the usage of notion homotopy for semigroup actions in this setting. We show that in the case when $I$ is a monoid the homotopical category $\act_l(I)$, the full-subcategory of $\bACT(I)$ of finite left $I$-sets, will admit $3$-arrow calculus, so that from 27.5 of \cite{dwyer} it will be saturated. As a result  we are able to define Burnside ring of a monoid $I$, which is another main objective of this note. We denote the Burnside ring of a monoid $I$ by $\A(I)$. This construction  generalizes  the Burnside ring of $\bbN$ given in \cite{yoshida} to any monoid, so that we can propose our construction for a  generalization of ``the theory of non-invertible finite dynamical systems" (which is a proposed problem in \cite{yoshida}, page 130) to ``theory of finite state automata". We also define analogues notion to the Burnside's mark homomorphism so that we get the theory of Burnside rings, see Theorem \ref{mark}. When $I$ is a group, our construction of Burnside ring agrees with the usual one existing in the literature, see \cite{tomdieck}, which is a very important construction in group theory and homotopy theory. If $I$ is a commutative monoid and $K(I)$ is its Gr\"othendieck group, then we have proved $\A(I)$  is equal to $\A(K(I))$, see Theorem \ref{burnside}. In particular we show $\A(\bbN)=\A(\bbZ)$, so that we can add one more arrow which would be an isomorphism in the main diagram in \cite{dress} page 3. 

We also recover the idea of the attractors for a finite state automaton (or attractors for non-linear finite dynamical systems in a generalized way) in analogy with the attractors in the field of dynamical systems, as the orbits of the reverse actions of a given monoid action. When the monoid is taken as $\bbN$, this will correspond to the standard definitions. The periodic attractors will be the generators of the Burnside ring. The Burnside ring of a free monoid on an alphabet is an invariant of the types of machines can be build, so that it would be very useful in Automata theory.

\section{Actions of semigroups on sets}\label{section:ActionsOnSets}

Given sets $A$ and $B$, we denote the set of functions from $A$ to $B$ by $[A,B]$ and we denote the set of endofunctions on $A$ by $\End (A)$. One can define two distinct monoid  structures on the set $\End (A)=[A,A]$, where the identity on $A$ is the identity of the monoid. In the first one we choose the monoid operation on $\End (A)$ as the composition of endofunctions when endofunctions are applied on $A$ from the right. Then we denote this monoid by $\End_r(A)$ and we write $fg$ for the composition of $f$ and $g$ in $\End_r(A)$, which we mean $f$ is applied first then $g$. In other words if $f$ and $g$ are in $\End_r(A)$ and $a$ is in $A$ then
$$(a)(fg)=((a)f)g.$$
Similarly, for the second one we write $\End_l(A)$ for the monoid obtained by taking the monoid operation on $\End (A)$ as the composition of endofunctions when endofunctions are applied on $A$ from the left. In this case we write $f\circ g$ for the composition of $f$ and $g$ in $\End_l(A)$. In other words if $f$ and $g$ are in $\End_l(A)$ and $a$ is in $A$ then
$$(f\circ g)(a)=f(g(a)).$$ 
We can also consider the endofunction sets  $\End_r(A)$ and  $\End_l(A)$ with the underlying semigroup structure.

\subsection{Actions on sets and function sets}\label{subsection:ActionsOnSetsAndFunctionSets}

Let $I$ be a semigroup (resp. a monoid). All through this section we denote the operation in $I$ by $\otimes $. Normally one defines an action of  $I$ on a set $A$ as a function $A\times I\rightarrow A $ which is compatible with the semigroup operation; or alternatively, it can be defined as a semigroup (resp. a monoid) homomorphism from $I$ to $\End_r(A)$ and call it a right action of $I$ on $A$. One can also define an action of a semigroup $I$ on a set $A$ as a semigroup (resp. a monoid) homomorphism from $I$ to $\End_l(A)$ and call it a left action of $I$ on $A$. However, in this paper we consider an action of a semigroup (resp. a monoid) on a set as a biaction. More precisely we have the following definition:
\begin{definition} \label{definition:ActionsOfsemigroups}
	Suppose that $I$ is a semigroup (resp. monoid) and $A$ is a set. An action $\alpha$ of $I$ on $A$ is a pair $(\alpha_l,\alpha_r)$ such that $\alpha_l:I\rightarrow \End_l(A)$ and $\alpha_r:I\rightarrow \End_r(A)$ are semigroup homomorphisms (resp. monoid homomorphism) and $\alpha_l$ commutes with $\alpha_r$ so that for all $i, j$ in $I$ and $a$ in $A$ we have $$(\alpha_l(i)(a))\alpha_r(j)=\alpha_l(i)((a)\alpha_r(j)).$$
	Instead of saying $\alpha$ is an action of $I$ on $A$, we could also say $\alpha $ is a $I$-action on $A$ or say $(A,\alpha )$ is a $I$-set or just say $A$ is a $I$-set.
\end{definition}
Suppose that we have $I$-actions $\alpha=(\alpha_l,\alpha_r)$ on $A$ and $\beta=(\beta_l,\beta_r)$ on $B$. There is an induced $I$-action
$$[\alpha ,\beta]=([\alpha ,\beta]_l,[\alpha ,\beta]_r)$$
on $[A,B]$ such that for $f$ in $[A,B]$ and $i$ in $I$ the function $[\alpha ,\beta]_l(i)(f)$ is the composition
\[\begin{tikzpicture}
\node (A) at (1,0) {$A$};
\node (B) at (3,0) {$A$};
\node (c) at (5,0) {$B$};
\node (d) at (7,0) {$B$};
\path[->,font=\scriptsize,>=angle 90]
(A) edge node[above] {$\alpha_r(i)$} (B)
(B) edge node[above] {$f$} (c)
(c) edge node[above] {$\beta_l(i)$} (d);
\end{tikzpicture}\]
and $(f)[\alpha ,\beta]_r(i)$ is the composition
\[\begin{tikzpicture}
\node (A) at (1,0) {$A$};
\node (B) at (3,0) {$A$};
\node (c) at (5,0) {$B$};
\node (d) at (7,0) {$B.$};
\path[->,font=\scriptsize,>=angle 90]
(A) edge node[above] {$\alpha_l(i)$} (B)
(B) edge node[above] {$f$} (c)
(c) edge node[above] {$\beta_r(i)$} (d);
\end{tikzpicture}\]

\subsection{Equivariant functions and fixed point sets}\label{subsection:EquivaraintFunctionsAndFixedPoints}

We first are going to define centralizers of semigroup and monoid actions. Let $(A,\alpha )$ be a $I$-set where $\alpha=(\alpha_l,\alpha_r)$. Then $C_A(I)$ the centralizer of $I$ in $A$ with the action $\alpha$ is defined as 
$$C_A(I)= \{a\in A :\forall  i    \in  I, \alpha_l(i)(a)=(a)\alpha_r(i)  \}.$$

Suppose that we have $I$-actions $\alpha=(\alpha_l,\alpha_r)$ on $A$ and $\beta=(\beta_l,\beta_r)$ on $B$. Considering the $I$-action $[\alpha ,\beta]$ on $[A,B]$ we define $\mathop{Map}_I(A,B)$ namely the set of $I$-equivariant functions from $A$ to $B$ as of $I$ in $[A,B]$ with the induced action $[\alpha ,\beta]$, i.e. 
$$\Map_I (A,B)=C_{[A,B]}(I).$$
Hence a function is called a $I$-equivariant function from $A$ to $B$ if it is in $\Map_I(A,B)$, so that a function $f:A\rightarrow B$ is a $I$-equivariant function if and only if  we have the identity
$$(f(\alpha_l(i)(a)))\beta_r(i)=\beta_l(i)(f((a)\alpha_r(i)))$$
for all $i$ in $I$ and $a$ in $A$.

Here we list some of the properties of equivariant functions similar to the classical case. Let $(A,\alpha )$,  $(B,\beta )$, $(C,\gamma )$, and $(D,\delta )$ be four $I$-sets. Assume $f:A\rightarrow B$ and $h:C\rightarrow D$ be two functions. The functions $f$ and $h$ induces a function $[B,C]\rightarrow [A,D]$ which sends $g:B\rightarrow C$ to $h\circ g\circ f$. The following result shows that compositions by equivariant functions induces an equivariant function between function sets.

\begin{proposition}\label{ind-funcsets}
	If $f:A\rightarrow B$ and $h:C\rightarrow D$ are two $I$-equivariant functions then the induced function $[B,C]\rightarrow [A,D]$ by $f$ and $h$ is $I$-equivariant.
\end{proposition}
\begin{proof}
	Since $f$ and $h$ are $I$-equivariant, we have
	$$(h(\gamma_l(i)(g(   (  f(  \alpha_l(i)(a)   )   )\beta _r(i)    ))))\delta _r(i)
	= \delta _l(i)(h((g(  \beta _l(i)(  f(  (a)\alpha _r(i)   )   )      ))\gamma _r(i)))$$
	for all $a$ in $A$, $i$ in $I$ and $g$ in $[B,C]$. Hence we have
	$$(h\circ (([\beta ,\gamma ]_l(i)(g))\circ f) )[\alpha,\delta ]_r(i) =
	[\alpha,\delta ]_l(i)(h\circ (((g)[\beta ,\gamma ]_r(i))\circ f) )$$
	for all $i$ in $I$ and $g$ in $[B,C]$. This means the induced function from $[B,C]$ to $[A,D]$ is $I$-equivariant.
\end{proof}

Let $A$ be a $I$-set. Then we define $\Fix_I(A)$ namely the set of fix points of $I$ on $A$ as  
$$\Fix_I(A)=\Map_I(*,A)$$
where $*$ denotes a set with one element and the trivial $I$-action on it.
\begin{proposition}
	Let $I$ be a semigroup or a monoid, and $A$, $B$ be two $I$-sets. Then we have a bijection
	$$\Map_I(A,B)\cong \Fix_I([A,B]).$$
\end{proposition}
\begin{proof}
	More generally for an $I$-set $A$ we have a bijection from $C_I(A)$ to $C_I([*, A])$ sending $z$ in $C_I(A)$ to the function from $*$ to $A$ which sends the unique point in $*$ to $z$.
\end{proof}
Given a function $f:A\rightarrow [B,C]$ we define $\bar{f}:A\times B \rightarrow C$ by $\bar{f}(a,b)=f(a)(b)$ for all $a$ in $A$ and $b$ in $B$.
\begin{proposition}
	Let $A$, $B$ and $C$ be three $I$-sets with $I$-actions $\alpha $, $\beta $ and $\gamma$ respectively. Then the function
	$$\Map_I(A,[B,C])\rightarrow \Map_I(A\times B, C)$$ defined by $f\mapsto \bar{f}$ is a bijection.
\end{proposition}
\begin{proof}
	We only need to show that $f:A\rightarrow [B,C]$ is a $I$-equivariant function if and only if $\bar{f}:A\times B \rightarrow C$ is a $I$-equivariant function.  We know that the statement $f:A\rightarrow [B,C]$ is a $I$-equivariant function means 	
	$$(f)[\alpha,[\beta,\gamma]]_r(i)=[\alpha,[\beta,\gamma ]]_l(i)(f)$$
	for all $i$ in $I$. In other words it means
	$$(f(\alpha_l(i)(a))(\beta_l(i)(b)))\gamma _r(i)=\gamma _l(i)(f((a)\alpha_r(i))((b)\beta_r(i)))$$
	for all $a$ in $A$, $b$ in $B$ and $i$ in $I$.  Hence it is equivalent to
	$$(\bar{f}(\alpha_l(i)(a),\beta_l(i)(b)))\gamma _r(i)=\gamma _l(i)(\bar{f}((a)\alpha_r(i),(b)\beta_r(i)))$$
	for all $a$ in $A$, $b$ in $B$ and $i$ in $I$. Therefore the statement $f:A\rightarrow [B,C]$ is a $I$-equivariant function is equivalent to $$(\bar{f})[\alpha \times \beta,\gamma]_r(i)=[\alpha \times \beta,\gamma]_l(i)(\bar{f})$$ which means $\bar{f}:A\times B \rightarrow C$ is a $I$-equivariant function.
\end{proof}
\begin{remark}\label{prod}
	 If $A$, $B$ and $C$ be three $I$-sets, then there is an obvious bijection
	 $$\Map_I(A, B\times C])\rightarrow  \Map_I(A ,B)\times\Map_I(A , C).$$
\end{remark}

\section{Categories of I-sets}\label{section:CategoriesofIsets}

Observe that when $I$ is a monoid, the usual categories of left (resp. right) actions of $I$ are just functor categories from $I$ to $\Set$, the category of sets. In this section, for a semigroup or a monoid $I$, we will define several categories whose class of objects are a subclass of the ``sets with an action of $I$" defined as in the sense of the previous section, so that they contains the usual category of left and right actions of $I$ as a full-subcategory. In each case the morphisms will be  $I$-equivariant functions defined according to the case being considered. In order to define objects of these categories we will first discuss semi-reversible actions and actions which are reversible on one side. Secondly, we will show that the composition of two $I$-equivariant functions  is an $I$-equivariant function under certain conditions. Finally we will give the definitions of categories of certain $I$-sets.

\subsection{Semi-reversible actions and actions reversible on one side}

Let $(A,\alpha )$ be an $I$-set. First note that if $\alpha_l(i)$  is an automorphism of $A$ then for all $a$ in $A$ then we have the equality
$$ 
\alpha_l(i)^{-1}((a)\alpha_r(j))=(\alpha_l(i)^{-1}(a))\alpha_r(j).
$$
and similarly in the case when $\alpha_r(i)$  is an automorphism of $A$ then we have
$$
\alpha_l(i)((a)\alpha_r(j)^{-1})=(\alpha_l(i)(a))\alpha_r(j)^{-1}.
$$

We say  $(A,\alpha )$  is ``semi-reversible" if either $\alpha _l(i)$ or $\alpha _r(i)$ is an automorphism of $A$ for all $i$ in $I$. We say $\alpha_l$ (resp. $\alpha_r$) is reversible if  $\alpha _l(i)$ (resp.  $\alpha_r(i)$ ) is an automorphism of $A$ for all $i$. A set $(A,\alpha )$  is called ``reversible on one side" if either $\alpha _l$ or $\alpha _r$ is reversible. Note that if an action is reversible on one side then it is semi-reversible. Hence the results about semi-reversible actions in this section are also true for actions that are reversible on one side.

\subsection{Compositions of equivariant functions}\label{subsection:CompositionofEquivFunc}

Compositions of equivariant functions may not be equivariant unless we have a semi-reversibility assumption in the following sense. Let $S$ be a set and $(B_s,\beta(s))$ be a semi-reversible $I$-set for $s$ in $S$. Define $B$ as the product $\Pi_{s\in S}B_s$ with the $I$-action given by $\beta (s)$ on the $s^{th}$ component. Assume $(A,\alpha )$ and $(C,\gamma )$ are $I$-sets and $f:A\rightarrow B$, $g:B\rightarrow C$ are $I$-equivariant functions. Then we have the following result
\begin{lemma}\label{comp}
	$g\circ f$ is $I$-equivariant.
\end{lemma}
\begin{proof}
	We want to show
	$$
	\begin{aligned}
	\gamma_l(i)((g\circ f)((a)\alpha_r(i)))
	&=((g\circ f)(\alpha_l(i)(a))) \gamma_r(i)&&\nonumber\\
	\end{aligned}
	$$
	for any $a$ in $A$ and $i$ in $I$.  Let us denote the left-hand side of above equality by $\operatorname{LHS}$ and the right-hand side by $\operatorname{RHS}$. Let $f_s$ denote the $s^{th}$ component of $f$. Given any $s$ in $S$ and $i$ in $I$, since $(B_s,\beta(s))$ is semi-reversible, there exists $x(s,i)$ in $\{l,r\}$ such that $\beta(s)_{x(s,i)}(i)$ is an automorphism of $B_s$. Since $\beta(s)_{x(s,i)}(i)^{-1}\circ \beta(s)_{x(s,i)}(i)$ is identity, we have
	$$
	\begin{aligned}
	\operatorname{LHS}
	&=\gamma_l(i)(g(f((a)\alpha_r(i)))&&\nonumber\\
	&=\gamma_l(i)(g((f_s((a)\alpha_r(i)))_{s\in S}))&&\\
	&=\gamma_l(i)(g((E(a,i)_s)_{s\in S}))&&\\
	\end{aligned}
	$$
	where
	$$	E(a,i)_s=
	\left\{ \begin{array}{cc}
	(\beta(s)_l(i)^{-1}\circ \beta(s)_l(i))(f_s((a)\alpha_r(i))) &\text{ if }x(s,i)=l \\
	(f_s((a)\alpha_r(i)))(\beta(s)_r(i)^{-1}\beta(s)_r(i)) &\text{ if }x(s,i)=r \\
	\end{array} \right.
	$$
	We have
	$$\operatorname{LHS}=\gamma_l(i)(g((F(a,i)_s)_{s\in S}))$$
	if $F(a,i)_s$ is defined as follows:
	$$
	F(a,i)_s=
	\left\{ \begin{array}{cc}
	\beta(s)_l(i)^{-1}(\beta(s)_l(i)(f_s((a)\alpha_r(i)))) &\text{ if }x(s,i)=l \\
	((f_s((a)\alpha_r(i)))\beta(s)_r(i)^{-1})\beta(s)_r(i) &\text{ if }x(s,i)=r \\
	\end{array} \right.
	$$
	Since $f$ is $I$-equivariant means $f_s$ is $I$-equivariant for all $s$ in $S$, we have
	$$\operatorname{LHS}=\gamma_l(i)(g((G(a,i)_s)_{s\in S}))$$ where
	$$
	G(a,i)_s=
	\left\{ \begin{array}{cc}
	\beta(s)_l(i)^{-1}((f_s(\alpha_l(i)(a)))\beta(s)_r(i)) &\text{ if }x(s,i)=l \\
	((f_s((a)\alpha_r(i)))\beta(s)_r(i)^{-1})\beta(s)_r(i) &\text{ if }x(s,i)=r \\
	\end{array} \right.
	$$
	By the above equality
	$$\operatorname{LHS}= (g(\beta(s)_l(i)((H(a,i)_s))_{s\in S})) \gamma_r(i)$$
	with
	$$
	H(a,i)_s=
	\left\{ \begin{array}{cc}
	\beta(s)_l(i)^{-1}(f_s(\alpha_l(i)(a))) &\text{ if }x(s,i)=r \\
	(f((a)\alpha_r(i)))\beta(s)_r(i)^{-1}&\text{ if }x(s,i)=r \\
	\end{array} \right.
	$$
	Since $g$ is $I$-equivariant
	$$
	\begin{aligned}
	\operatorname{LHS}
	&=(g((J(a,i)_s)_{s\in S})) \gamma_r(i)&&\nonumber\\
	&=(g((K(a,i)_s)_{s\in S})) \gamma_r(i)&&\\
	\end{aligned}
	$$
	where
	$$
	J(a,i)_s=
	\left\{ \begin{array}{cc}
	\beta(s)_l(i)(\beta(s)_l(i)^{-1}(f_s(\alpha_l(i)(a)))) &\text{ if }x(s,i)=l \\
	\beta(s)_l(i)((f_s((a)\alpha_r(i)))\beta(s)_r(i)^{-1})&\text{ if }x(s,i)=r \\
	\end{array} \right.
	$$
	and
	$$
	K(a,i)_s=
	\left\{ \begin{array}{cc}
	f_s(\alpha_l(i)(a)) &\text{ if }x(s,i)=l \\
	(\beta(s)_l(i)(f_s((a)\alpha_r(i))))\beta(s)_r(i)^{-1}&\text{ if }x(s,i)=r \\
	\end{array} \right.
	$$
	Since $f_s$ is $I$-equivariant for all $s\in S$, we have
	$$
	\begin{aligned}
	\operatorname{LHS}
	&=(g((L(a,i)_s)_{s\in S})) \gamma_r(i)&&\nonumber\\
	&=(g((f_s(\alpha_l(i)(a)))_{s\in S}) \gamma_r(i)&&\\
	&=(g(f(\alpha_l(i)(a))) \gamma_r(i)&&\\
	&=\operatorname{RHS}&&\\
	\end{aligned}
	$$
	where
	$$
	L(a,i)_s=
	\left\{ \begin{array}{cc}
	f_s(\alpha_l(i)(a)) &\text{ if }x(s,i)=l \\
	((f_s(\alpha_l(i)(a)))\beta(s)_r(i))\beta(s)_r(i)^{-1}&\text{ if }x(s,i)=r \\
	\end{array} \right.
	$$
	This completes the proof.
\end{proof}

\begin{proposition}\label{isom}
	Let $f:A\to B$ be a bijective equivariant function where $(A,\alpha)$ and $(B,\beta)$ are semi-reversible finite $I$-sets. Then the inverse $f^{-1}$ is equivariant.
\end{proposition}
\begin{proof}
	Assume $f$ is equivariant. We want to show $$\alpha_l(i)(f^{-1}((b)\beta_r(i))) =(f^{-1}(\beta_l(i)(b)))\alpha_r(i).$$
	Assume first both $\alpha_l(i)$ and $\beta_l(i)$ is isomorphism. First since both $f$ and $\alpha_l(i)$ are bijective, we can write	$$\beta_l(i)((b)\beta_r(i))=(f(\alpha_l(i)\alpha_l(i)^{-1}(f^{-1}(\beta_l(i)(b))))\beta_r(i). $$
	Since $f$ is equivariant,
	$$\beta_l(i)((b)\beta_r(i))=\beta_l(i)(f(\alpha_l(i)^{-1}((f^{-1}(\beta_l(i)(b)))\alpha_r(i)))). $$
	and since $\beta_l(i)$ is bijective, we get
	$$(b)\beta_r(i) =f(\alpha_l(i)^{-1}((f^{-1}(\beta_l(i)(b)))\alpha_r(i))) $$ which implies $$\alpha_l(i)(f^{-1}((b)\beta_r(i))) =(f^{-1}(\beta_l(i)(b)))\alpha_r(i).$$
	The case when both $\alpha_r(i)$ and $\beta_r(i)$ is isomorphism is similar. Assume now  $\alpha_r(i)$ and $\beta_l(i)$ are isomorphisms. Since $f$ is an isomorphism, the composition of $f^{-1}$, $\alpha_r(i)$ and $\beta_l(i)$ is an isomorphism. Since $A$ and $B$ are finite sets, from the equality 
	$$(f(\alpha_l(i)(a)))\beta_r(i)=\beta_l(i)(f((a)\alpha_r(i)))$$
	we get $\alpha_l(i)$ and $\beta_r(i)$ are isomorphisms as well. Hence,  $f^{-1}$ is equivariant. The case  $\alpha_r(i)$ and $\beta_l(i)$ are isomorphism is the same. Hence this proves the statement.
\end{proof}
Observe that if the semi-reversible actions are isomorphism in the same side, then we do not need the finiteness assumption. However, in general this proposition is not correct when we drop the assumption on finiteness. For example if $I=\bbN$ and $A=B=\bbN$ with the actions $\alpha=(\alpha_l,1)$ on $A$ such that $\alpha_l(1)(i)=i+1$ and $\beta=(1,\beta_r)$ on $B$ such that $\beta_r(1)(i+1)=i$ and $\beta_r(1)(0)=0$, then the identity function $id:A\to B$ is equivariant but $id:B\to A$ is not.

\subsection{Definitions of categories of $I$-sets}\label{subsection:EquivalenceOfViewPoints}

Let $I$ be a semigroup, considering the usual definition one-sided of actions we let $\ACT_l(I)$,  $\ACT_r(I)$, $\act_l(I)$, $\act_r(I)$ to denote the category of left $I$-sets, right $I$-sets, finite left $I$-sets and finite right $I$-sets respectively, with $I$-equivariant maps. Now we define four new categories denoted by $\ACT(I)$, $\act(I)$, $\bACT(I)$ and $\bact(I)$. The objects of the categories $\ACT(I)$ and $\act(I)$ are $I$-sets which are products of semi-reversible $I$-sets and finite $I$-sets which are products of semi-reversible $I$-sets respectively, where $I$-sets are defined as in the previous section. The objects of $\bACT(I)$ and $\bact(I)$ are $I$-sets which are products of sets with actions that are reversible on one side and finite $I$-sets which are products of  sets with actions that are reversible on one side respectively, where again  $I$-sets are defined as in the previous section. The morphisms of the categories $\ACT(I)$, $\act(I)$, $\bACT(I)$, $\bact(I)$ are $I$-equivariant functions (defined as in Section \ref{subsection:EquivaraintFunctionsAndFixedPoints}).

For a semigroup (or monoid) we have the following diagram  
\[\begin{tikzpicture}
\node (0) at (4,6) {$\ACT(I)$};
\node (1) at (6,6) {$\bACT(I)$};		\node (11) at (2,6) {$\act(I)$};
\node (1r) at (7,4) {$\ACT_r(I)$};	\node (1l) at (4,4) {$\ACT_l(I)$};
\node (011) at (1,4) {$\bact(I)$};
\node (2l) at (2,2) {$\act_l(I)$};		\node (2r) at (6,2) {$\act_r(I)$};
\path[->,left hook-latex,font=\scriptsize,>=angle 90]
(1) edge node[above]  { } (0)
(2r) edge node[above] { } (011)
(2l) edge node[left] { } (011);
\path[->,right hook-latex,font=\scriptsize,>=angle 90]
(11) edge node[right] { } (0)
(1r) edge node[left] { } (1)
(2r) edge node[above] { } (1r)
(1l) edge node[above] { } (1)	
(011) edge node[above]  { } (11)
(011) edge node[right] { } (1)
(2l) edge node[left] { } (1l);
\end{tikzpicture}\]
so that all of the functors are inclusions, which map an $I$-set to itself.

\section{Action reversing functors}\label{section:InverseActionFunctor}

For a semigroup $I$ we define four semigroup homomorphisms as follows: The homomorphisms
$$\iota_l:I\rightarrow \End_l(I)\text{ \ \ and \ \ }\iota_r:I\rightarrow \End_r(I)$$
sends every element to the identity endofunction and the homomorphisms
$$\mu_l:I\rightarrow \End_l(I)\text{ \ \ and \ \ } \mu_r:I\rightarrow \End_r(I)$$
are given by multiplication from the left and the right, respectively.

\subsection{Reversing actions from left to right}\label{subsection:InverseActionsOnSets}
Consider $I$ itself as a $I$-set with the action $(\iota_l,\mu_r)$. Let $A$ be a set with a $I$-action $\alpha$. To indicate the right action on $I$ is trivial let us denote the set of equivariant functions, $\Map_I(I,A)$, by $\lMap (A)$. Let $f:I\rightarrow A$ be a $I$-equivariant map, i.e., for every $i,j$ in $I$ we have $$(f(j))\alpha_r(i)=\alpha_l(i)(f(j\otimes i))$$
We define a $I$-action $\theta=(\theta _l,\theta_r)$ on $ \lMap (A)$ as follows: The left component
$$\theta _l:I\rightarrow \End_l(\lMap (A))$$
sends an element $k$ in $I$ to the function
$$\theta _l(k):\lMap (A)\rightarrow \lMap (A)$$
defined as the identity function. Hence the function $\theta_l(k)$ sends $f$ to $f$. The right component
$$\theta_r:I\rightarrow \End_r(\lMap (A))$$
sends an element $k$ in $I$ to the function
$$\theta_r(k):\lMap (A)\rightarrow \lMap (A)$$ defined as the function that  sends $f$ to the composition
\[\begin{tikzpicture}
\node (A) at (1,0) {$I$};
\node (B) at (3,0) {$I$};
\node (c) at (5,0) {$A$};

\path[->,font=\scriptsize,>=angle 90]
(A) edge node[above] {$\mu_l(k)$} (B)
(B) edge node[above] {$f$} (c);
\end{tikzpicture}\]
so that we have $(f)\theta_r(k)(j)=f(k\otimes j)$, for every $j,k \in I$. Since $I$ is semi-reversible, by Lemma \ref{comp} we can say $\theta$ is well defined.

We call this action the reverse (from left to right) action of $\alpha$. In fact this construction is functorial on $\ACT(I)$ and we denote the functor sending an $I$-action on a set $A$ to the reverse $I$-action on $\lMap (A)$  by $$\lMap :\ACT(I)\rightarrow \ACT(I).$$ This functor sends a morphism $f:A\rightarrow B$ to the morphism which sends $h:I\rightarrow A$ to the composition $f\circ h$ from $I$ to $B$. Given $I$-set $A$ we can define the evaluation function $$\evl:\lMap (A)\to A$$ given by $\evl(f)=f(1)$ whenever we have $1$.  

\begin{lemma}\label{evl}
	$\evl$ defines a natural transformation from $ \lMap $ to $id$, the identity functor.
\end{lemma}
\begin{proof}Let  $A$ be an $I$-set with action $\alpha$. From the equality	$$\alpha_l(i)(\evl((f)\theta_r(i)))=\alpha_l(i)(f(i))=(f(1))\alpha_r(i)=(\evl(f))\alpha_r(i) ,$$
we can say $\evl$ is equivariant, so that it defines a natural transformation from $ \lMap $ to $id$.
\end{proof}

\subsection{Reversing actions from right to left}

We can also define reverse actions from right to left. This time we consider $I$  as an $I$-set with the action $(\mu_l,\iota_r)$, so that an $I$-equivariant function   $f:I\rightarrow A$ satisfies
$$(f(i\otimes j)) \alpha_r(i)=\alpha_l(i)(f(j))$$
for every $i,j$ in $I$. In this case we denote the set of equivariant functions from $I$ to $A$,  $\Map_I(I,A)$, by $\rMap (A)$. We define a $I$-action $\vartheta=(\vartheta _l,\vartheta_r)$ on $ \rMap (A) $ as follows: The left component
$$\vartheta_l:I\rightarrow \End_l(\rMap (A))$$
sends an element $k$ in $I$ to the function
$$\vartheta_l(k):\rMap (A)\rightarrow \rMap (A)$$ defined as the function that  sends $f$ to the composition
\[\begin{tikzpicture}
\node (A) at (1,0) {$I$};
\node (B) at (3,0) {$I$};
\node (c) at (5,0) {$A$};

\path[->,font=\scriptsize,>=angle 90]
(A) edge node[above] {$\mu _r(k)$} (B)
(B) edge node[above] {$f$} (c);
\end{tikzpicture}\]
so that we have $\vartheta_l(k)(f)(i)=f(i\otimes k)$.
The right component
$$\vartheta_r:I\rightarrow \End_r(\rMap (A))$$
sends an element $k$ in $I$ to the function
$$\vartheta_r(k):\rMap (A)\rightarrow \rMap (A)$$ defined as the identity function. Hence the function $\vartheta_r(k)$ sends $f$ to $f$. Again by Lemma \ref{comp} this construction is well defined. There is again an equivariant evaluation function $$\evr:\rMap (A)\to A$$ given by $\evr(f)=f(1)$ provided that we have $1$, which is equivariant. Similar to the Lemma \ref{evl}, $\evr$ defines a natural transformation from $ \rMap $ to $id$.

The following is an important property of reverse actions when $I$ is a monoid:
\begin{proposition}\label{leftinv}
	Let $(A,\alpha )$ be an $I$-set such that the right action $\alpha_r$ is reversible, then there is a isomorphism $\rMap (A)\cong  A$ as $I$-sets. If  the left action $\alpha_l$ is reversible, then there is a isomorphism $\lMap (A)\cong  A$ as $I$-sets. % Moreover, this set .
\end{proposition}

\begin{proof}
	Assume $\alpha_r$ is reversible . Define a map $\phi:A\to \rMap$	such that $\phi(a)=f_a$ for $a\in A$  where $$f_a(i)=\alpha_l(i)(a)\alpha_r(i)^{-1}.$$ This map is well defined since $$f_a(i\otimes j)\alpha_r(i)=\alpha_l(i\otimes j)(a)\alpha_r(i\otimes j)^{-1}\alpha_r(i)= \alpha_l(i)(f_a( j)),$$
	$f_a$ is in $\rMap$. Since $f_a(1)=a$, $\phi$ is the inverse of the $\evr$, so that $\evr$ is a bijection and by Proposition \ref{isom} $\phi$ is equivariant, so that we get an isomorphism of $I$-sets. The proof for the case $\alpha_l$ is reversible is the same.
\end{proof}

\subsection{As idempotent endofunctors on $\bACT(I)$}

Let $I$ be a monoid. The following lemma shows that the reversing functors are idempotent.
\begin{theorem}\label{idempotent}
	The evaluations function $\evl$ (resp. $\evr$) defines a natural isomorphism from $\lMap\circ \lMap$ to $\lMap$ (resp.  from $\rMap\circ \rMap$ to $\rMap$). 
\end{theorem}
\begin{proof}
	For any $I$-set $A$, consider the function $$\Phi_A:\lMap(A)\to \lMap\circ \lMap(A)$$ given by $$\Phi(g)(i)(j)=g(i\otimes j)$$ for $g\in \lMap(A)$ and $i,j\in I$. It is straightforward to check the equalities  
	$$\alpha_l(k)(\Phi(g)(i)(j\otimes k))=(\Phi_A(g)(i)(j))\alpha_l(k)$$ and $$\Phi(g)(i\otimes k)=(\Phi(g)(i))\theta_r(k)$$ 
	so that $\Phi$ is well defined. Since $$g(k\otimes i\otimes j) =\Phi((g)\theta_r(k))(i)(j)=(\Phi(g))\theta_r(k)(i)(j)=g(k\otimes i\otimes j),$$
 $\Phi$ is equivariant. For  any  $g\in \lMap(A)$ we have
	$$(\evl\circ\Phi)(g)(i)=\Phi(g)(i)(1)=g(i)$$ and for any $h \in \lMap\circ \lMap(A)$ we have $$(\Phi\circ \evl)(h)(i)(j) =\Phi(h(1))(i)(j) =h(1)(i\otimes j)=h(i)(j)$$ 
	so that $\evl$ and $\Phi$ are mutual inverses. This completes the proof. The same proof works for $\evr$ as well.
\end{proof}

We denote the composition of two reverse endo-functors on $\bACT(I)$ by $\InvE$, in other words we have
$$\InvE=\lMap \circ \rMap$$
considered as an endofunctor on $\bACT(I)$. An equivariant function $f$ in $\InvE (A)$ satisfies $$f(i\otimes j)(i\otimes k)=f(j)(k)$$ for every $i, j$ and $k$ in $I$. For any $I$-set $A$ we have an evaluation function
$$\ev :\InvE(A)\rightarrow A$$ 
defined by  $\ev (f)=f(1)(1)$. If $\gamma$ is the inverse of the inverse action on $A$, i.e. action on $\rMap (\lMap (A))$, then we have $$(\ev(\gamma_l(i)(f)))\alpha_r(i)=(f\circ \mu_r(i)(1)(1))\alpha_r(i)=(f (i)(1))\alpha_r(i)$$by equivariance of $f(i)$ this is equal to  $$\alpha_l(i)(f(i)(i))=\alpha_l(i)(f(1)(1))=\alpha_l(i)(\ev(f))$$
hence, $\ev$ is equivariant. Then  $\ev$ defines a natural transformation from $\InvE\circ \InvE$ to $\InvE$. When $I$ is a commutative monoid, we have the following proposition:
\begin{proposition} 
	If $I$ is a commutative monoid then $\ev$ defines a natural isomorphism from $\InvE\circ \InvE$ to $\InvE$. 
\end{proposition}
\begin{proof}
	For any $I$-set $A$, the function $$\Phi_A:\InvE(A)\to \InvE\circ \InvE(A)$$ given by $$\Phi_A(g)(i)(j)(k)(l)=g(i\otimes k)(j\otimes l)$$ for $g\in \InvE(A)$ and $i,j,k,l \in I$. It is straightforward to check that this function is equivariant since on both $\InvE(A)$ and $\InvE\circ \InvE(A)$, the right actions are trivial. We have  $$\ev (\Phi_A(g))(k)(l) =g(k)(l)$$ and $$\Phi_A(\ev (g))(k)(l) =g(k)(l)$$ so that $\ev $ and $\Phi_A$ are mutual inverses. This completes the proof.
\end{proof}

\subsection{Reverse actions on finite sets}\label{Section:invfinite}
We again use the same notations for the restrictions of  $\lMap$, $\rMap$ and their compositions $\InvE$ on $\bact(I)$. Let $(A,\alpha )$ be an $I$-set such that the right action $\alpha_r$ is trivial. For an element $a$ in $A$, let  $Ia$ denote the orbit set $$Ia=\{ \alpha_l(i)(a):i\in I\}$$ and for a given $f:I\to A$ in $\lMap(A)$ let $If(I)$ denote the set $$If(I)=\{ \alpha_l(i)(f(j)):i,\ j\in I\}.$$
We define a set $ A^l$ as the set
$$ A^l=\{a\in A:\ \text{for}\ \text{all}\ i \in I,\ \alpha_l(i)|_{Ia} \text{ is one-to-one}\}.$$
Note that 
\begin{lemma}\label{invfinite}
	Let $I$ be a monoid and let $A$ be a finite set. Let $(A,\alpha )$ be an $I$-set such that the right action $\alpha_r$ is trivial. Then there is a isomorphism $\lMap (A)\cong  A^l$ as $I$-sets.% Moreover, this set .
\end{lemma}
\begin{proof}
	Firstly, for an element $a \in  A^l$ we define $f_a:I\to A$ with $f_a(i)=\alpha_l(i)^{-1}(a)$, then since $a\in  A^l$, this is a well-defined map. By definition for every $i,j$ in $I$ we have	
	$$\alpha_l(i)f_a(j\otimes i)=\alpha_l(i)\alpha_l(j \otimes i)^{-1}(a)=\alpha_l(j)^{-1}(a)=f_a(j).$$ Hence $f_a$ is equivariant and we have an injective function $ A^l\to \lMap (A)$.
	
	Now suppose that $f:I\to A$ be a function in $\lMap (A)$. %Claim: $f(1) \in  A^l$.
	We claim that $f(1)$ is an element of $ A^l$. Assume the contrary that there exist $i,j, k$ in $I$ such that $$\alpha_l(j)(f(1))\neq \alpha_l(k)(f(1))\ \text{ and }\  \alpha_l(i\otimes j)(f(1))= \alpha_l(i\otimes k)(f(1)).$$ Since  $A$ is finite then for every  $i\in I$ there exist positive integers $m, m'$ with $m>m'$ such that for all $a$ in $If(I)$, we have the identity $\alpha_l(i^m)(a)=\alpha_l(i^{m'})(a)$. Hence the restriction of $\alpha_l(i^{m-m'})$ to the set $$\alpha_l(i^{m'})(If(I)):=\{\alpha_l(i^{m'})(a):a\in If(I)\} $$ is the identity function. Moreover,  for any $v\in I$ we have
	$$f(v)=\alpha_l(i^{m'})f(v\otimes i^{m'} )$$
	so that $im(f)$ is contained in $\alpha_l(i^{m'})(If(I)).$ %	 so that  $\alpha_l(i^{m-m'})$ is identity on $im(f)$ as well.
	
	Let $j$ and $k$ be two elements in $I$. As above there are integers $t,t'$ with $t>t'$ and   $$\alpha_l(j^{t'})f(j^{t})=\alpha_l(j^{t'})f(j^{t'})  $$ so that   $\alpha_l(j)(f(1))=f(j^{t-t'-1})$.  Similarly there are integers $s,s'$ with $s>s'$ and   $ \alpha_l(k)(f(1))=f(k^{s-s'-1}). $ Hence both  $ \alpha_l(k)(f(1)) $ and  $ \alpha_l(k)(f(1)) $ are elements of $im(f)$, which means $\alpha_l(i^{m-m'})$ is identity on both.
	
	By our initial assumption we have
	$$ \alpha_l(i^{m-m'-1})(\alpha_l(i\otimes j)(f(1)))=\alpha_l(i^{m-m'-1})(\alpha_l(i\otimes k)(f(1)))$$
	which implies	
	$$ \alpha_l(i^{m-m' })(\alpha_l(  j)(f(1)))=\alpha_l(i^{m-m' })(\alpha_l(  k)(f(1)))$$
	As a result we get  $ \alpha_l(  j)(f(1))= \alpha_l(  k)(f(1))$, i.e. a contradiction, so that $f(1)$ must be an element of $ A^l$. The evaluation function $\evl$ is injective by definition of  $A^l$ and $\evl (f_a)=a$. By Proposition \ref{isom}  we  get an isomorphism as desired. This completes the proof.
\end{proof}
Objects in $\bact(I)$ are the actions with either left or right component is reversible. Assume $A$ is an $I$-set with right component is reversible. Then we define $A^l$ as  $\rMap A^l$. We have the following lemma:
\begin{lemma}\label{invfinitebact}
 There is an isomorphism $\lMap (A)\cong  A^l$ as $I$-sets.% Moreover, this set .
\end{lemma}
\begin{proof}
The proof follows from Lemma \ref{invfinite} and Proposition \ref{leftinv}.
\end{proof}

For an $I$-set $A$ we define
$ A^r $ similarly. We have a similar lemma as follows:
\begin{lemma}\label{invfiniter}
Let $(A,\alpha )$ be an $I$-set such that the left action $\alpha_l$ is reversible.	Then there is an isomorphism $\rMap (A)\cong  A^r$.% Moreover, this set .
\end{lemma}  
Let $\bev$ denote the restriction of $\ev$ on finite $I$-sets. Note that $\bev$ is bijective by the previous propositions. We have the following lemma:
\begin{proposition} \label{isominve}
$\bev$ defines a natural isomorphism from $\InvE\circ \InvE$ to $\InvE$. 
\end{proposition}
\begin{proof}

	This proposition directly follows from Proposition \ref{isom}, since $\bev$ from $\InvE\circ \InvE$ to $\InvE$ is bijective, by the Lemma \ref{invfinite} and Proposition \ref{leftinv}.
\end{proof}

\section{Equivalence of view points on groups}

The following theorem shows that Definition \ref{definition:ActionsOfsemigroups} is equivalent to the usual one for groups.
\begin{theorem} For a group $G$, the categories $\act(G)$, $\bact(G)$, $\act_l(G)$ and $\act_r(G)$  are all equivalent to each other as categories and $\ACT(G)$, $\bACT(G)$, $\ACT_l(G)$, $\ACT_r(G)$ are all equivalent to each other as categories.
\end{theorem}
\begin{proof}
	Here we will only prove the equivalence of $\ACT(G)$ and $\ACT_l(G)$ the rest is either similar or just obtained by restrictions of the equivalences. First note that the functor $$\rMap:\ACT(G)\to \ACT(G)$$ factors through the inclusion $$inc: \ACT_l(G)\to\ACT(G).$$
	We again write $$\rMap:\ACT(G)\to \ACT_l(G)$$  for the functor in the factorization, by an abuse of notation. Then this functor sends an object $(A,\alpha)$ in $\ACT(G)$ to the left action $\mu :G\rightarrow \End_l(A)$ given by
	$$\mu(g)(a)= \alpha_l(g)((a)\alpha_r(g^{-1}))$$
	and sends a morphism $f$ from $(A,\alpha)$ to $(B,\beta)$ to itself considered as a function from $A$ to $B$. Now clearly $\rMap\circ inc$ is identity on $\ACT_l(G)$. By Proposition \ref{evl} and \ref{leftinv}, $\evr$ defines a natural isomorphism from $inc\circ \rMap$ to the identity on $ACT(G)$. Hence, this gives an equivalence between $\ACT(G)$  and $\ACT_l(G)$.
\end{proof}

We define a functor $$\operatorname{inv}_l^r:\ACT_l (G)\rightarrow \ACT_r (G) $$ which sends a left $G$ action $$\nu:G\times A\rightarrow A\text{, given by }(g,a)\mapsto g.a $$
for $g\in G$ and $a\in A$, to a right $G$-action
$$\nu^{-1}:A\times G\rightarrow A\text{, given by }(a,g)\mapsto g^{-1}.a$$ for $g\in G$ and $a\in A$, i.e. the reverse action of $\nu$. The following theorem shows that the two definitions we gave for reverse actions agree for group actions.

\begin{theorem}\label{l2r}
	The diagram
	\[\begin{tikzpicture}
	\node (0) at (1,0) {$\ACT(G)$};
	\node (c) at (4,0) {$\ACT(G)$};
	\node (a) at (1,1.5) {$\ACT_l(G)$};
	\node (b) at (4,1.5) {$\ACT_r(G)$};
	\path[->,font=\scriptsize,>=angle 90]
	(a) edge node[above]  {$\operatorname{inv}_l^r$} (b)
	(b) edge node[right] {$inc$} (c)
	(a) edge node[left] {$inc$} (0)
	(0) edge node[above] {$\lMap $} (c);
	\end{tikzpicture}\]
	is commutative up to a natural isomorphism.
\end{theorem}
\begin{proof}
This follows from Proposition \ref{leftinv}, since group actions are reversible on both sides.
\end{proof}

A version of Theorem \ref{l2r} is also true for the case of reversing actions from right to left, i.e. the diagram
\[\begin{tikzpicture}
\node (0) at (1,0) {$\ACT(G)$};
\node (c) at (4,0) {$\ACT(G)$};
\node (a) at (1,1.5) {$\ACT_r(G)$};
\node (b) at (4,1.5) {$\ACT_l(G)$};
\path[->,font=\scriptsize,>=angle 90]
(a) edge node[above]  {$\operatorname{inv}_r^l$} (b)
(b) edge node[right] {$inc$} (c)
(a) edge node[left] {$inc$} (0)
(0) edge node[above] {$\rMap $} (c);
\end{tikzpicture}\]
is commutative up to a natural isomorphism, where $\operatorname{inv}_r^l$ is defined similarly.

\section{Homotopy category of monoid actions and the Burnside ring}\label{homotoical}
In this section we discuss homotopical category structure on $\bACT(I)$ where $I$ is a monoid. We refer to \cite{dwyer} for general terminology and homotopical notions in this section. Let $A, B$ be $I$-sets in $\bACT(I)$ and $f:A\to B$ be an $I$-equivariant map. We say $f$ is a weak equivalence if the induced function $\InvE(f):\InvE (A)\to \InvE (B)$ is an isomorphism. We denote the class of weak equivalences by $\WW$. It is straightforward to check that these weak equivalences satisfy the $2$-out-of-$6$ property, since isomorphisms do. Hence this makes  $\bACT(I)$ into a homotopical category. The homotopical structure on the subcategories of $\bACT(I)$ is defined accordingly. 

In order to define the Burnside ring of a monoid $I$ we concentrate on the actions of $I$ on finite sets. Note that the functor $$\InvE:\bact(I)\to \bact (I)$$ factors through the inclusion $$inc: \act_l(I)\to\bact(I).$$ We again denote the functor $\bact(I)\to \act_l(I)$ in  the factorization by $\InvE$, by an abuse of notation. Note that the functor $$\InvE:\bact(I)\to \act_l(I)$$ preserve weak equivalences so does the inclusion $$inc: \act_l(I)\to \bact(I).$$ The composition $\InvE\circ inc$ is identity functor on $ \act_l(I)$ and there is a natural weak equivalence from $inc\circ \InvE$ to $id_{ \bact(I)}$ given by the evaluation map $\bev$. Hence $\act_l(I)$ is a left deformation retract of $\bact(I)$, so that their homotopy categories are naturally equivalent (see \cite{dwyer}, 26.3, 26.5 and 29.1). We will continue with the category  $\act_l(I)$ to define the Burnside ring. The category $\act_l(I)$ has nice properties such as monomorphisms are stable under pushouts and epimorphisms are stable under pullbacks \cite{topos}, as it is a topos, so that isomorphisms are also stable under pullbacks and pushouts. In fact assume we have a diagram 
\[\begin{tikzpicture}[scale=.8]
\node (0) at (1,0) {$A$};
\node (c) at (3,0) {$C$};
\node (a) at (1,2) {$D$};
\node (b) at (3,2) {$B$};
\path[->,font=\scriptsize,>=angle 90]
(a) edge node[above]  {$f'$} (b)
(b) edge node[right] {$p$} (c)
(a) edge node[left] {$\imath$} (0)
(0) edge node[above] {$f$} (c);
\end{tikzpicture}\]
Pullbacks and pushouts are given in a standard way. If $D$ is the pullback of the maps $p$ and $f$ where $\alpha$ , $\beta$ and $\gamma$ are the actions on $A$, $B$ and $C$ respectively, then $D$ is given as the set $$D=\{(a,b)\in A\times B: f(a)=p(b)\}$$ and the action $\delta$ on $D$ is given by pair of actions, i.e. $\delta_l=(\alpha_l,\beta_l)$ and trivial right action. The maps $\imath$ and $f'$ are induced by projections so that they are equivariant.

If the above square is a pushout then $$C=(A\amalg B)/\sim $$ where  $\imath(d)\sim f'(d)$ for all $d$ in $D$. The action $\gamma$ on $C$ is given by 
	$$
	\gamma_l(i)[x]=
	\left\{ \begin{array}{cc}
 \alpha_l(i)(x)  &\text{ if }x\in A \\
	\beta_l(i)(x) &\text{ if }x\in B \\
	\end{array} \right.
	$$
for all $i\in I$, with trivial right action. By equivariance of the maps $\imath$ and $f'$ in diagram, so that for all $d\in D$ and $i\in I$ we have $ \alpha_l(i)(\imath(d)) =  \imath(\delta_l(i)(d))$ and $\beta_l(i)(f'(d))= f'(\delta_l(i)(d))$, so that $\alpha_l(i)(\imath(d)) \sim \beta_l(i)(f'(d))$, i.e. the action is well defined. The maps $p$ and $f$ are induced by inclusions so that they are also equivariant. 

We will show that the category $\act_l(I)$ admits a $3$-arrow calculus, for details of $3$-arrow calculus we refer to \cite{dwyer}, 27.3.

\subsection{Saturation of the category $\act_l(I)$}
Let us denote the homotopy category of  $\act_l(I)$  by  $\Ho(\act_l(I))$ and let $L:\act_l(I)\to \Ho(\act_l(I))$ be the localization with respect to the above weak equivalences (see \cite{dwyer} 26.5). We will show that $\act_l(I)$ admits a $3$-arrow calculus. To do this we define two subclasses $\UU$ and $\VV$ of the class weak equivalences $\WW$ of  $\act_l(I)$ as follows: $\UU$ will be the subclass  of $\WW$ which are also inclusions and $\VV$ will be the subclass  of $\WW$ which are also surjections. Firstly, suppose that we have a zig-zag $A'\stackrel{u}\leftarrow A\stackrel{f}\to B$ in  $\act_l(I)$ where $u$ is in $\UU$. Then we can associate another zig-zag $A'\stackrel{f'}\to B'\stackrel{u'}\leftarrow  B$ from the pushout
\[\begin{tikzpicture}[scale=.8]
\node (0) at (1,0) {$A'$};
\node (c) at (3,0) {$B'$};
\node (a) at (1,2) {$A$};
\node (b) at (3,2) {$B $};
\path[->,font=\scriptsize,>=angle 90]
(a) edge node[above]  {$f$} (b)
(b) edge node[right] {$u'$} (c)
(a) edge node[left] {$u$} (0)
(0) edge node[above] {$f'$} (c);
\end{tikzpicture}\]
so that $f'\circ u=u'\circ f$ and the function $u'$ is an inclusion. Let $\alpha$, $\alpha'$, $\beta$ and $\beta'$ be the actions on $A$, $A'$, $B$ and $B'$ respectively. Since right actions are trivial, to be able to see $u'$ is weak equivalence, it is enough to show $\lMap(B') $ is contained in the image of  $\lMap(u') $. Assume the contrary and let $\sigma:I\to B'$ be a map in  $\lMap(B') $ which is not in the image of  $\lMap(u') $. Then $\sigma(1)$ is not in the image of $u'$ because otherwise $\sigma$ factors through $u'$ since $\sigma(1)\in (B')^l$ (see Lemma \ref{invfinite}), so that  $\sigma(1)$ is in the image of $f'$. Thus, there is an element $a'$ in $A'$ such that $f'(a')=\sigma(1)$. Assume first $a'\notin (A')^l$ i.e. there exist $i,i_1, i_2$ in $I$ such that $$\alpha'_l(i_1)(a')\neq \alpha'_l(i_2)(a')\ \text{but} \ \alpha'_l(i\otimes i_1)(a')= \alpha'_l(i\otimes i_2)(a')$$
then there exist  $b\in B$ such that $u'(b)=f'(\alpha'_l(i_1)(a'))$. But as in the proof of Lemma \ref{invfinite} there exist an integer $m$ such that $$f'(a')=\alpha'_l(i_1^m)(f'(\alpha'_l(i_1)(a')))=\alpha'_l(i_1^m)(u'(b))=u'(\beta'_l(i_1^m)(b)),$$
but then this leads us a contradiction unless $a'\in (A')^l$, so that $\sigma$ must be an element in the image of $\lMap(f')$. Since $u$ is a weak equivalence, any element in $\lMap(A')$ factors through $u$, so that $\sigma$ is in the image of $\lMap(f'\circ u)$. But then we get a contradiction again since $\sigma$ is not in the image of $\lMap(u'\circ f)$. Hence, $u'$ is a weak equivalence, i.e. $u'$ is in $\UU$. If $u$ is an isomorphism then $u'$ is also an isomorphism since both $u$ and $u'$ fits in above pushout diagram.

Similarly if we have a zig-zag  $X\stackrel{g}\to Y\stackrel{v}\leftarrow  Y'$  in   $\act_l(I)$ where $v$ is in $\VV$, then we can associate another zig-zag $X\stackrel{v'}\leftarrow X'\stackrel{g'}\to Y$ from the pullback diagram
\[\begin{tikzpicture}[scale=.8]
\node (0) at (1,0) {$X$};
\node (c) at (3,0) {$Y$};
\node (a) at (1,2) {$X'$};
\node (b) at (3,2) {$Y'$};
\path[->,font=\scriptsize,>=angle 90]
(a) edge node[above]  {$g'$} (b)
(b) edge node[right] {$v$} (c)
(a) edge node[left] {$v'$} (0)
(0) edge node[above] {$g$} (c);
\end{tikzpicture}\]
so that $g\circ v'=v\circ g'$, and the function $v'$ is a surjection. Let $\sigma:I\to X'$ , $\bar{\sigma}:I\to X'$ elements in $\lMap(X') $ with $\sigma(i)=(x_i,y_i)$ and $\bar{\sigma}(i)=(\bar{x}_i,y_i)$ for $i\in I$ $x_i,\bar{x}_i \in X$ and $y\in Y'$, i.e. $\lMap(v')({\sigma})=\lMap(v')(\bar{\sigma})$. Since $\lMap(g')({\sigma})(i)=x_i$ and $\lMap(g')(\bar{\sigma})(i)=\bar{x}_i$, we have $\lMap(v)(x_i)=\lMap(v')(y_i) =\lMap(v)(\bar{x}_i)$. We know $v$ is a weak equivalence so that $\lMap(v)$ is bijection, thus $x_i=\bar{x}_i$, i.e. $v'$ is a weak equivalence. Hence $v'$ is in $\VV$. Again if $v$ is an isomorphism then so does $v'$ since both fits into a pullback diagram.

Assume now $w:M\to N$ is a weak equivalence in  $\act_l(I)$, then consider the pushout diagram
\[\begin{tikzpicture}[scale=.8]
\node (0) at (1,0) {$N$};
\node (c) at (4.5,0) {$M'$};
\node (a) at (1,2) {$\InvE(M) $};
\node (b) at (4.5,2) {$M$};
\path[->,font=\scriptsize,>=angle 90]
(a) edge node[above]  {$\bev$} (b)
(b) edge node[right] {$u$} (c)
(a) edge node[left] {$w\circ \bev$} (0)
(0) edge node[above] {$\tilde u$} (c);
\end{tikzpicture}\]
 From the Lemmas \ref{invfinite} and \ref{invfiniter} we know $\bev$ is injective. Since the above square is a pushout, $\tilde u$ is injective. Hence, there is a unique function $v:M'\to N$ which is surjective. As before, the functions $u$ and $v$ are also equivariant, so that we have a factorization of $w$ as $w=v\circ u$ such that $v$ is in $\VV$ and $u$ is in $\UU$. Hence $\act_l(I)$ admits a $3$-arrow calculus $\{\UU,\VV\}$. Then by 27.5 of \cite{dwyer} we can conclude that $\act_l(I)$ is saturated, i.e. a function in $\act_l(I)$ is a weak equivalence if and only if its image in $\Ho(\act_l(I))$, under the localization functor, is an isomorphism.
 
Note that it is possible to define stronger classes of weak equivalences on these categories which still make them homotopical categories, by using similar ideas above along with restrictions of actions to submonoids or subsets. However, not all of them admit a $3$-arrow calculus. For a given a submonoid $J$ of $I$ let  $\Res_J^I:\act_l(I)\to \act_l(J)$ be the restriction functor, which sends a finite $I$-set $(A,\alpha)$ to the $J$-set $A$ with the restriction of the action $\alpha$ on $J$. Let $Z$ be a collection of submonoids of $I$ which contains $I$. A function $f:A\to B$ in $\act_l(I)$ is called a $Z$-equivalence if for every $J$ in $Z$ the function $\InvE(\Res_J^I(f) )$ is an $I$-equivariant isomorphism. Since $\Res_J^I$ respects compositions, the class of $Z$-equivalences satisfies both $2$-out-of-$3$ and $2$-out-of-$6$ properties, and so that again $\act_l(I)$ with $Z$-equivalences will be a homotopical category admitting a $3$-arrow calculus, when we set $\UU$ as the subclass of $Z$-equivalences which are inclusions and $\VV$ as the subclass of $Z$-equivalences which are surjections. It is now straightforward to check that these classes satisfied the required axioms. A $Z$-equivalence is trivially a weak equivalence so that $Z$-equivalences are stronger form of weak equivalences. This is a possible direction to look but it is too complicated. However, in this paper we continue with the weak equivalences instead of $Z$-equivalences for convenience.

\subsection{Burnside ring}
In the classical theory of group actions, when a group $G$ is given, the Burnside ring of $G$, denoted by $A(G)$, is defined as the Gr\"othendieck ring of the semiring of isomorphism classes of finite $G$-sets where the addition is given by disjoint union and the multiplication is given by cartesian product. The Burnside ring of a group is a very important construction in group theory, and has several applications, see e.g. \cite{tomdieck}, \cite{dress}, \cite{dress2}, \cite{dress3}.  We define the Burnside ring of a monoid by using the homotopical structure on $\act_l(I)$. The isomorphism classes in $\Ho(\act_l(I))$ forms a semiring under disjoint union as addition and cartesian product as multiplication. We call the Gr\"othendieck ring of this semiring as the Burnside ring of $I$, and we denote this ring by $\A(I)$. Most of the properties of this Burnside ring follows from Section \ref{Section:invfinite}.

By definition the Burnside ring of a group given in this way is equal to the standard one. Hence, it does validate the name ``the Burnside ring of a monoid".  Moreover, the following proposition shows that the definitions of the Burnside ring of a commutative monoid is same as the Burnside ring of its Gr\"othendieck construction. Let us denote by $K(I)$ the Gr\"othendieck group of a commutative monoid $I$. Then $\A(K(I))$ denotes the usual Burnside ring of the group  $ K(I) $ (see e.g. \cite{tomdieck}). 
\begin{theorem}\label{burnside}
If $I$ is commutative monoid then $\A(I)$ is isomorphic to $\A(K(I))$.
\end{theorem}
\begin{proof}
Define $\widetilde\Lambda:\act_l(K(I))\to \act_l(  I )$ induced  by the natural map from $I$ to $K(I)$ and let $\Lambda:\A(K(I))\to   \A(I)$ denote the induced function on Burnside rings. Here we will define the inverse of $\Lambda$. Let $A$ be an $I$-set with action $\alpha$ and let $\vartheta$ be the action on $\InvE(A)$.  Lemma \ref{invfinite} implies that the action on $\InvE(A)$ has a group action factorization, i.e. the map  $ \vartheta_l:I\to \End_l(\InvE(A))$ factors through the inclusion $ \Aut_l(\InvE(A))\hookrightarrow \End_l(\InvE(A))$. Hence, we can consider $\InvE(A)$ as a $K(I)$-set. Define a function  $\Gamma:\A(I)\to \A(K(I))$ by sending a class $[A]$ of $I$-set $A$ in $\A(I)$ to the class $[\InvE(A)]$ in $ \A(K(I))$. Notice that  $\InvE(\widetilde\Lambda(A))\cong\widetilde\Lambda(A)$ by Proposition \ref{leftinv}, so that $\Gamma\circ \Lambda$ is identity. The composition  $\Lambda\circ \Gamma$ is also identity since by Proposition \ref{isominve}, the natural transformation $\bev$ gives a weak equivalence from $ \InvE(A)$ to $A$. Hence $\Gamma$ is a ring isomorphism with the inverse $\Lambda$. %Since for every $I$-set $M,N$ we have $$  \lMap (M\times N)\cong \lMap (M)\times\lMap (N)$$ and $$  \lMap (M\amalg N)\cong \lMap (M)\amalg\lMap (N)$$	then these maps are ring isomorphisms.
\end{proof}

\begin{remark}
	Note that one can also define Burnside ring with $Z$-equivalences on $\act_l(I)$ defined in the previous section, which again will coincide with the definition of Burnside ring of a group. However, in this case for an arbitrary monoid the Burnside ring would be much bigger and would have a very complicated structure, so that the classification problems would become very difficult. 
\end{remark}

\subsubsection{Burnside mark homomorphism}
Assume $I$ is a monoid and $A$ is a finite left $I$-set. Let $J$ be a submonoid of $I$. We define the mark $\hat m_J(A)$ of $J$  on $A$ as the number of elements in $\InvE(A)$ that are fixed by every element in $I$, i.e. if $\vartheta$ is the action on $\InvE(A)$ (which is the action obtained by reversing $\alpha$ twice) then  $$\hat m_J(A)=|\Fix_J(\InvE(A))|.$$ In other words, $\hat m_J(A)$ is the number of equivariant functions in $\InvE(A)$ satisfying $ f(i\otimes j)(k)=f(i)(k)$ for every $i$, $k$ in $I$ and $j$ in $J$. This defines a semiring homomorphism $$\hat m_J:\Isom( \Ho(\act_l(I)))\to \bbZ$$ since $$\InvE(A \amalg  B)=\InvE(A)\amalg \InvE(B),$$ so that $\hat m_J(A\amalg B)=\hat m_J(A)+\hat m_J(B) $ and $$\InvE(A \times  B)=\InvE(A)\times \InvE(B)$$
and hence $ \hat m_J(A\times B)=\hat m_J(A). \hat m_J(B) $ same as the usual case. The associated ring homomorphism $$  m_J:\A(I)\to \bbZ$$
is called the mark homomorphism at $J$. Note that when a finite $I$-set $A$ is given, the  image of $ \vartheta_l$ in $\Aut_l(\InvE(A))$ form a subgroup, and let $\vartheta_l(I)$ denote this subgroup. Let $\vartheta_l(J)$ denote the image of the submonoid $J$ under $\vartheta_l $, which is a subgroup of $\vartheta_l(I)$. Then mark of $A$ at $J$ corresponds to the usual mark of $\vartheta_l(I)$ at the subgroup $\vartheta_l(J)$. We call an $I$-set $(A,\alpha)$ in $\act_l(I)$ weakly-transitive, if for every pair of elements $f, g\in \InvE(A)$, there is an elements $i$ in $I$ such that $\vartheta_l(i)(f)=g$. The set $\InvE(A)$ can be expressed as a disjoint union of orbits  $\vartheta_l(I)/\vartheta_l(J_t)$  for some family $ \{J_t:t\in T \} $ of submonoids. Hence isomorphism classes weakly transitive $I$-sets generate the additive group of Burnside ring, same as the classical case, see \cite{tomdieck-rep}. 

Let $J$ and $J'$ be two submonoids of $I$. We say $J$ and $J'$ are weakly conjugate, and we write $J\conj J'$, if for every $I$-set $A$ the subgroup $\vartheta_l(J)$ is conjugate to $\vartheta_l(J')$ in  $\vartheta_l(I)$, where  $\vartheta_l$ is the action on $\InvE(A)$. It is straightforward to check that weak conjugation is an equivalence relation. Let $Y(I)$ be the set of weak conjugacy classes of $I$, i.e. the set of equivalence classes of  `$ \conj  $'. Observe that weakly conjugate submonoid have the same mark, i.e. if $J\conj J'$ then  for any given $I$-set $A$ we have $ m_J(A)= m_{J'}(A) $, which follows from standard group theory facts. Hence one can see the mark homomorphism as a ring homomorphism $$ m : \A(I)\to \bigoplus_{[J]\in Y(I)} \bbZ$$ into the direct sum of integers $\bbZ$,  so that $\displaystyle m=\bigoplus_{[J]\in Y(I)} m_J$.

\begin{theorem}\label{mark}
	The mark homomorphism $m$ is injective.
\end{theorem}
\begin{proof}
Proof follows from ideas of the proof in the standard case (see \cite{tomdieck-rep}, Proposition 1.2.2). Let $A$ be an $I$-set and $x$ be the corresponding element in the Burnside ring, and let $\vartheta$ be the action on $\InvE(A)$. Then, since there is an induced action of the group  $\vartheta_l(I)$, we can write $$x=\sum_{J\in Y(I)}z_t[ \vartheta_l(I)/\vartheta_l(J)] $$ with $z_J \in \bbZ$. Let $K$ be a monoid such that $\vartheta_l(K)$ is the maximal conjugacy class in $\vartheta_l(I)$ with respect to the inclusion, with $z_K\neq 0$. The rest is same as the proof of Proposition  1.2.2 in \cite{tomdieck-rep}. Since $\Fix_{K}( \vartheta_l(I)/\vartheta_l(J))$ is non-empty if and only if $\vartheta_l(K)$ is sub-conjugate to  $\vartheta_l(J)$, we have $m(x) $ is non-zero due to maximality of  $\vartheta_l(K)$. Hence  $m$ is injective.
\end{proof}

\subsection{Attractors of finite state automata} 
It is well known that the notion of attractors and attracting sets play an important role in physics, geometry and in particular in the theory of  dynamical systems, see for example \cite{seibert}, \cite{milnorattr}, \cite{bonifant}. We define analogues notion for finite state automata. We will consider an automaton in the following sense. Let $I$ be a free semigroup on an alphabet and $X$ be an $I$-set with action $(\alpha_l,1)$, i.e. the right action is trivial. Assume both $I$ and $X$ has a topology and the action is continuous. 

\begin{definition}
	An attractor $A$ of this action is a subset of $X$  defined by the following properties: There exists neighborhood of $X$ in $X$ denoted by $B(X)$ called a basin of attraction for $A$ such that for all submonoid $I'$ of $I$  we have
	$$A=\coprod_{s\in S}A_s$$
	for some index set $S$  so that the following holds for every $s$ in $S$:
	\begin{enumerate}
		\item $A_s$ is forward invariant, i.e. $I'A_s\subset A_s$.
		\item For every $b$ in $B(A)$ there is a word $w$ such that
		$\alpha_l(w)(b)$ is in $A_s$.
		\item $A_s$ is a minimal set satisfying the above two properties.
	\end{enumerate}
\end{definition}

This definition can be viewed as a special case of the definition given in \cite{souza}, when we consider attractors as global uniform attractors in the basin of attraction. In this paper we only use discrete topology on sets. We will say that an attractor is periodic if it is invariant up to isomorphism under the endofunctor $\InvE$. Notice that the Burnside ring of $I$ is generated by periodic attractors for finite $I$-sets (which are same as the orbits of the $I$-sets in the image of $\InvE$) with discrete topology. Hence, the Burnside ring can be used to understand types of periodic attractors which is also important for the usual definition of periodic attractors. 

As some simple examples, consider the set with two elements $\{x_0,x_1\}$ with the actions of free monoid with two generators $\{i_0,i_1\}$ as follows

$$
	
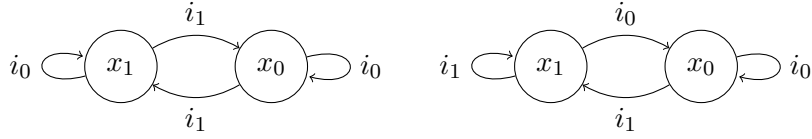
\begin{tikzpicture}[auto,node distance=2cm]
 \node[state] (x1)      {$x_1$};
 \node[state]         (x2) [right of=x1] {$x_0$};

 \path[->] (x1)  edge [loop left] node {$i_0$} (x1)
 edge  [bend left]            node {$i_1$} (x2)

 (x2) edge [bend left]  node {$i_1$} (x1)
 (x2)
 edge [loop right] node {$i_0$} (x2)
 ;
\end{tikzpicture}
\hspace{.5cm}
\begin{tikzpicture}[auto,node distance=2cm]
  \node[state] (x1)      {$x_1$};
  \node[state]         (x2) [right of=x1] {$x_0$};

  \path[->] (x1)  edge [loop left] node {$i_1$} (x1)
  edge  [bend left]            node {$i_0$} (x2)

  (x2) edge [bend left]  node {$i_1$} (x1)
  (x2)
  edge [loop right] node {$i_0$} (x2)
  ;
\end{tikzpicture}$$
In the first case (on the left side) the reverse action will be isomorphic to itself, and the attractor is the set itself; however, in the second case, on the right, the inverse action is empty so that there is no attractor. If we consider the action on $\{x_0,x_1\}$ with the action of free monoid with one  generator $i$, given by 
$$
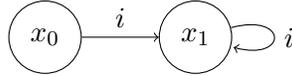
\begin{tikzpicture}[auto,node distance=2cm]
\node[state] (x1)      {$x_1$};
\node[state]         (x2) [left of=x1] {$x_0$};

\path[->] (x1)  edge [loop right] node {$i$} (x1)
(x2) edge  node {$i$} (x1);
\end{tikzpicture}$$
then the reverse  action will be singleton with trivial action on it, so the attractor in this case is $\{x_1\}$. 

As a last remark we should note that Lemma \ref{invfinite} is not valid in the case when $A$ is infinite. For example let $I\cong\bbN$ with a single generator $i$ and the set $A$ and the action of $I$ be as in the figure below

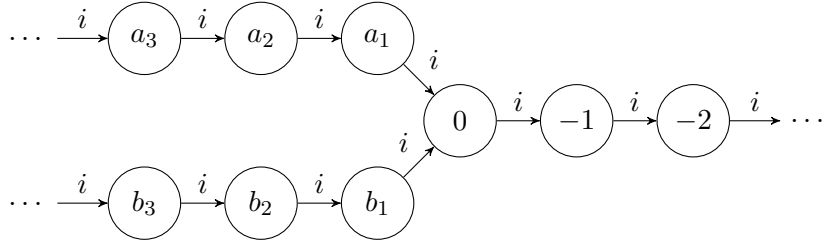
\[\begin{tikzpicture}[auto,>=stealth',node distance=1.55cm]
\node[state] (x0)      			{$0$};
\node[state] (x1) [above left of=x0] {${a_{1}}$};
\node[state] (y1) [below left of=x0] {${b_{1}}$};
\node[state] (x2) [left  of=x1] {${a_{2}}$};
\node[state] (y2) [left  of=y1] {${b_{2}}$};
\node[state] (x3) [left  of=x2] {${a_{3}}$};
\node[state] (y3) [left  of=y2] {${b_{3}}$};

\node[state] (-1) [right of=x0] {${-1}$};

\node[state] (-2) [right  of=-1] {${-2}$};

\node (vd) [right  of=-2] {$\dots$};
\node (vdx) [left  of=x3] {$\dots$};
\node(vdy) [left  of=y3] {$\dots$};

\path[->] (x1) edge  node {$i$} (x0)
(y1) edge  node {$i$} (x0)
(y2) edge  node {$i$} (y1)
(y3) edge  node {$i$} (y2)
(x2) edge  node {$i$} (x1)
(x3) edge  node {$i$} (x2)
(x0) edge  node {$i$} (-1)
(-1) edge  node {$i$} (-2)
(-2) edge  node {$i$} (vd)

(vdx) edge  node {$i$} (x3)
(vdy) edge  node {$i$} (y3);
\end{tikzpicture}\]
Then the reverse action will be as follows:
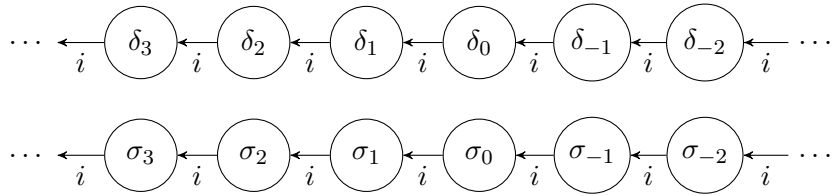
\[\begin{tikzpicture}[auto,>=stealth',node distance=1.5cm]
\node[state] (x0)      			{$\sigma_{0}$};
\node[state] (y0) [above  of=x0] {$\delta_{0}$};
\node[state] (x1) [left  of=x0] {$\sigma_{1}$};
\node[state] (y1) [left  of=y0] {$\delta_{1}$};
\node[state] (x2) [left  of=x1] {$\sigma_{2}$};
\node[state] (y2) [left  of=y1] {$\delta_{2}$};
\node[state] (x3) [left  of=x2] {$\sigma_{3}$};
\node[state] (y3) [left  of=y2] {$\delta_{3}$};

\node[state] (-1x) [right of=x0] {$\sigma_{{-1}}$};
\node[state] (-2x) [right  of=-1x] {$\sigma_{{-2}}$};

\node[state] (-1y) [right of=y0] {$\delta_{{-1}}$};
\node[state] (-2y) [right  of=-1y] {$\delta_{{-2}}$};

\node (vx) [right  of=-2x] {$\dots$};
\node (vy) [right  of=-2y] {$\dots$};
\node (vdx) [left  of=x3] {$\dots$};
\node(vdy) [left  of=y3] {$\dots$};

\path[->] (x0) edge  node {$i$} (x1)
(y0) edge  node {$i$} (y1)

(y1) edge  node {$i$} (y2)
(y2) edge  node {$i$} (y3)
(x1) edge  node {$i$} (x2)
(x2) edge  node {$i$} (x3)
(-1y) edge  node {$i$} (y0)
(-2y) edge  node {$i$} (-1y)
(vy) edge  node {$i$} (-2y)

(-1x) edge  node {$i$} (x0)
(-2x) edge  node {$i$} (-1x)
(vx) edge  node {$i$} (-2x)
(x3) edge  node {$i$} (vdx)
(y3) edge  node {$i$} (vdy);
\end{tikzpicture}\]
where the maps $\sigma_{k}$ and $\delta_{k}$ in the inverse action are defined by 
	$$
	\sigma_{{k}}(i^n)=
	\left\{ \begin{array}{cc}
	 a_{n+k}   &\text{ if }n+k>0 \\
	n+k &\text{ if } n+k\leq 0 \\
	\end{array} \right.
	$$
and
	$$
	\delta_{{k}}(i^n)=
	\left\{ \begin{array}{cc}
	b_{n+k}   &\text{ if }n+k>0 \\
	n+k &\text{ if } n+k\leq 0. \\
	\end{array} \right.
	$$

%\begin{acknowledgement}	This is a post-peer-review, pre-copyedit version of an article published in Applied Categorical Structures. The final authenticated version is available online at: http://dx.doi.org/10.1007/s10485-016-9477-4\end{acknowledgement}

\bibliographystyle{plain}
\bibliography{res2}
\end{document}